\documentclass[12pt]{amsart}
\usepackage{amsmath, amsthm, amscd, amsfonts}
\usepackage{xcolor}

\newtheorem{theorem}{Theorem}[section]
\newtheorem{lemma}[theorem]{Lemma}

\theoremstyle{definition}
\newtheorem{definition}[theorem]{Definition}
\newtheorem{example}[theorem]{Example}

\theoremstyle{remark}

\newtheorem{remark}[theorem]{Remark}
\numberwithin{equation}{section}

\newfont{\kh}{msbm10}

\begin{document}
\title[Linear maps preserving Kasparov cycles]
{Linear maps preserving Kasparov cycles and the characterization of induced automorphisms}

\author{K. Sharifi}
\address{Kamran Sharifi, \newline Department of Mathematics,
Shahrood University of Technology, P. O. Box 3619995161-316,
Shahrood, Iran} \email{sharifi.kamran@gmail.com}


\begin{abstract}
Let $E$ be a Hilbert C*-module and $\mathcal{L}(E)$ the C*-algebra of all bounded adjointable operators on $E$.
An operator $T \in \mathcal{L}(E)$
is a Kasparov cycle if $T^*T - 1$ and $TT^*-1$ are compact operators on $E$.
If $\mathcal{L}(E)\rightarrow \mathcal{L}(E)$ is a prime C*-algebra, and
$\varphi:\mathcal{L}(E)\rightarrow \mathcal{L}(E)$ is a linear map
which is unital and surjective up to compact operators,
and preserves Kasparov cycles in both directions,
then the induced map
$\psi:\mathcal{L}(E)/\mathcal{K}(E) \rightarrow \mathcal{L}(E)/\mathcal{K}(E)$
is either a $*$-automorphism or a $*$-anti-automorphism.
Our work extends the main result of
[\textit{J. Math. Anal. Appl.} \textbf{354} (2009), 625-629] to the set of Kasparov cycles,
showing that the assumption ``$\mathcal{L}(E)/\mathcal{K}(E)$ has real rank zero''
is redundant in our results.
\end{abstract}
\maketitle


\section{Introduction and Preliminaries}
In KK-theory and K-homology, one needs to study unitary operators modulo compact operators,
or what are now known as \emph{Kasparov cycles}. These were first introduced by Brown, Douglas,
and Fillmore \cite{BDF}, Kasparov \cite{KAS}, and more recently by
Bunke and Hirschmann \cite{bunkeh}, and the author \cite{SharAtiy} in
the study of K-functors, extensions of C*-algebras, and index theory.
In this paper, we characterize linear maps $\varphi$ on the C*-algebra of
all bounded adjointable operators $\mathcal{L}(E)$ that preserve Kasparov cycles.
Since the set of Kasparov cycles is a subset of $\mathcal{A}$-Fredholm operators,
our results can be viewed as modifications of the work of Aghasizadeh and
Hejazian \cite{AghaHejazi} to the context of Kasparov cycles. Our findings are also of
independent interest in the framework of Hilbert spaces.

A Jordan isomorphism from an algebra $\mathcal{A}$ to an algebra $\mathcal{B}$ is a
bijective linear map $\varphi$ satisfying $\varphi(T^2)=(\varphi(T))^2$ for every $T \in \mathcal{A}$.
Suppose $\mathcal{B}(H)$ is the C*-algebra of all bounded linear operators on a
Hilbert space $H$. It is well-known that every Jordan automorphism of $\mathcal{B}(H)$
is either an inner automorphism or an inner anti-automorphism (recall that a bijective
linear map on $\mathcal{B}(H)$ is called an anti-automorphism
if $\varphi(ST) = \varphi(T)\varphi(S)$ for all $S, T \in \mathcal{B}(H)$),
but the converse is not true in general. Linear preserver problems ask whether
assuming $\varphi : \mathcal{B}(H) \to \mathcal{B}(H)$ is a bijective linear
map having a certain preserving property implies that such a map must be a
Jordan automorphism or of a similar canonical form. One of the most famous
problems in this direction is Kaplansky's problem \cite{Kaplan}, asking whether
bijective unital linear maps preserving invertibility in both directions are Jordan automorphisms.

The first result on invertibility or spectrum-preserving linear maps is
implicit in the work of Dieudonn\'e \cite{DIEU}. An invertible linear
map $\varphi: \mathcal{M}_n(\mathbb{C}) \to \mathcal{M}_n(\mathbb{C})$
preserves the set of singular matrices if and only if there
exist $n \times n$ invertible matrices $P$ and $Q$ such that
either $\varphi(T) = PTQ$ or $\varphi(T) = PT^t Q$ for all $T$,
where $\mathcal{M}_n(\mathbb{C})$ denotes the C*-algebra of $n \times n$ complex
matrices and $T^t$ denotes the transpose of $T$ with respect to a fixed orthonormal basis.
Jafarian and Sourour extended this result to operator algebras on Banach spaces \cite{Jaf1}.
The problem of characterizing linear maps that preserve the spectrum of each operator,
or part of the spectrum, and structural properties has attracted the attention of many mathematicians;
see, e.g., \cite{BM, Cui1, HouCui1, Jaf1, Mbekhta1, Mbekhta2}.
Frank et al. \cite{FMP, FRMed} and Aghasizadeh and Hejazian \cite{AghaHejazi} have
investigated structure-preserving mappings on Hilbert C*-modules.
In this paper, we study linear maps on $\mathcal{L}(E)$ that preserve Kasparov
cycles on Hilbert C*-modules. This investigation enables us to characterize
linear maps that preserve the spectrum of modular operators.

Let $\mathcal{A}$ be a (possibly non-unital) C*-algebra, and let $E$ be a right Banach $\mathcal{A}$-module
equipped with an $\mathcal{A}$-valued inner product $\langle \cdot, \cdot \rangle : E \times E \to \mathcal{A}$
compatible with the complex structures on $\mathcal{A}$ and $E$, satisfying:
$$ \langle x, y \rangle = \langle y, x \rangle^* \quad \text{and} \quad
\langle x, x \rangle \geq 0,$$
with equality if and only if $x = 0$.
Here $\langle \cdot, \cdot \rangle$ is conjugate-linear in the first variable.
We consider modules that are complete under the norm $\|x\| := \|\langle x,x \rangle\|_{\mathcal{A}}^{1/2}$.
A Hilbert $\mathcal{A}$-submodule $E$ of $F$ is an orthogonal
summand if $F=E \oplus E^\bot$,
where $E^\bot :=\{ y \in F : \langle x,y \rangle = 0 \text{ for all } x \in E \}$.
We denote by $\mathcal{L}(E)$ the C*-algebra of all adjointable operators on $E$,
i.e., maps $T : E \to E$ for which there exists $T^* : E \to E$ with
$\langle Tx, y \rangle = \langle x, T^* y \rangle$ for all $x,y \in E$.
For $x,y \in E$, the operator $\theta_{x,y}(z) = x \langle y,z \rangle$
is bounded and adjointable, with $\theta_{x,y}^* = \theta_{y,x}$.
The closed linear span of $\{ \theta_{x,y} : x,y \in E \}$ in $\mathcal{L}(E)$
is the $*$-ideal of compact operators, denoted by $\mathcal{K}(E)$. The quotient algebra $\mathcal{L}(E)/\mathcal{K}(E)$
is the Calkin algebra of $E$, which coincides with the corona algebra
of $\mathcal{K}(E)$. Closed submodules of Hilbert modules need not be orthogonally complemented at all,
but Lance \cite[Theorem 3.2]{LAN} gives conditions under which they may be. The theory of
adjointable operators on Hilbert C*-modules provides a fundamental framework for operator
algebras and K-theory; see \cite{Black, LAN}.

Let $\pi : \mathcal{L}(E) \to \mathcal{L}(E)/\mathcal{K}(E)$ be the natural quotient map.
For $a \in \mathcal{L}(E)$ or $\mathcal{L}(E)/\mathcal{K}(E)$, the left spectrum $\sigma_l(a)$,
right spectrum $\sigma_r(a)$, and spectrum $\sigma(a)$ are defined as:
\begin{eqnarray*}
\sigma_l(a) &:=& \{ \lambda \in \mathbb{C} : \lambda 1 - a \text{ is not left invertible} \},\\
\sigma_r(a) &:=& \{ \lambda \in \mathbb{C} : \lambda 1 - a \text{ is not right invertible} \}, ~{\rm and}\\
\sigma(a) &:=& \{ \lambda \in \mathbb{C} : \lambda 1 - a \text{ is not invertible} \}.
\end{eqnarray*}
We define the set of \emph{Kasparov cycles} on $E$ by
$$ \mathcal{KC}(E) := \{ T \in \mathcal{L}(E) : \pi(T) \text{ has inverse } \pi(T^*) \}.$$
Thus $T \in \mathcal{L}(E)$ is a Kasparov cycle if and only if
$T^*T = 1$ mod $\mathcal{K}(E)$ and $TT^* = 1$ mod $\mathcal{K}(E)$, where $1$
is the identity of $\mathcal{L}(E)$, i.e., the identity operator $I$ on $E$.
Such operators arise naturally in Kasparov's KK-theory and K-homology \cite{Black}.

\begin{example}\label{Kasparovcycle}
Let $\mathcal{A}$ be a unital C*-algebra and let $H_{\mathcal{A}}$ be the standard
Hilbert $\mathcal{A}$-module countably generated by the orthonormal basis
$\xi_j = (0,\dots,0,1,0,\dots), j \in \mathbb{N}$. Define $T$ on $H_{\mathcal{A}}$ by
$T(\xi_1, \xi_2, \dots) = (0, \xi_1, \xi_2, \dots)$. Then $T \in \mathcal{KC}(H_{\mathcal{A}})$
since $T^*T = I$ and $TT^* - I$ is of finite rank.
\end{example}

The set $\mathcal{KC}(E)$ consists of $\mathcal{A}$-Fredholm operators \cite{bunkeh, SharAtiy},
and our results do not follow from those of Aghasizadeh and Hejazian \cite{AghaHejazi}. Indeed,
the main result of \cite{AghaHejazi}, Theorem 1.2, asserts that if $\varphi:\mathcal{L}(E) \to \mathcal{L}(E)$
is a unital, surjective linear map preserving semi-$\mathcal{A}$-Fredholm operators in both directions,
and $\mathcal{L}(E)/\mathcal{K}(E)$ is a prime C*-algebra with real rank zero,
then $\varphi(\mathcal{K}(E)) \subseteq \mathcal{K}(E)$ and the induced
map $\tilde{\varphi} : \mathcal{L}(E)/\mathcal{K}(E) \to \mathcal{L}(E)/\mathcal{K}(E)$, $\tilde{\varphi}(\pi(T)) = \pi(\varphi(T))$,
is an automorphism or anti-automorphism. In this paper, we extend this result to Kasparov cycles,
showing that the real rank zero assumption on $\mathcal{L}(E)/\mathcal{K}(E)$ is redundant.
Specifically, if $\varphi:\mathcal{L}(E) \to \mathcal{L}(E)$ is linear, unital, and surjective
up to compact operators, then $\sigma(\pi(\varphi(T))) = \sigma(\pi(T))$ for all $T \in \mathcal{L}(E)$
if and only if $\varphi$ preserves Kasparov cycles in both directions,
if and only if $\varphi(\mathcal{K}(E)) \subseteq \mathcal{K}(E)$ and the induced
map $\psi : \mathcal{L}(E)/\mathcal{K}(E) \to \mathcal{L}(E)/\mathcal{K}(E)$, $\psi(\pi(T)) = \pi(\varphi(T))$,
is a continuous $*$-automorphism or $*$-anti-automorphism.

When $E=H$ is a Hilbert space, the Calkin algebra $\mathcal{B}(H)/\mathcal{K}(H)$ is prime,
so our results align with and extend those of Mbekhta \cite{Mbekhta1, Mbekhta2} and Hou--Cui \cite{HouCui1}.

Throughout the rest of this paper, we use $1_{\mathcal{A}}$, or simply $1$,
to denote the identity of a unital C*-algebra $ \mathcal{A}$,
$I$ for the identity operator on a Hilbert module, and $Ker( \cdot )$ for the kernel of an operator.

\section{Linear maps preserving kasparov cycles}
In recent years, there has been significant interest in studying linear maps on
operator algebras that preserve specific properties of operators, cf. \cite{BM, FRMed} and references therein.
In this section, we will characterize the
linear maps on $\mathcal{L}(E)$ which preserves Kasparov cycles in both directions.
Throughout this section, we assume always that  $E$ is a Hilbert module over an arbitrary C*-algebra $ \mathcal{A}$ and
$\pi:\mathcal{L}(E)\rightarrow \mathcal{L}(E)/\mathcal{K}(E)$ denotes the natural quotient map.
To prove our main results, we require the following lemma and remarks.

\begin{lemma}\label{Ulemma1}
Let $C \in \mathcal{L}(E)$, then $C$ is a compact operator if and only if $T+C \in \mathcal{KC}(E)$, for every
$T \in \mathcal{KC}(E)$.
\end{lemma}

\begin{proof}
Let $C$ be a compact operator on $E$ and let $T$ be in  $\mathcal{KC}(E)$, then
$\pi(T+C)=\pi(T)$ is a unitary element in $\mathcal{L}(E)/\mathcal{K}(E)$. That is, $T+C \in \mathcal{KC}(E)$.

Conversely, let $C \in \mathcal{L}(E)$ and $T+C \in \mathcal{KC}(E)$, for every
$T \in \mathcal{KC}(E)$. Since both $I$ and $-I$ belong to $\mathcal{KC}(E)$,
both $I + C$ and $-I + C$ are in $\mathcal{KC}(E)$. Applying the hypothesis
that $T + C \in \mathcal{KC}$ for all $T \in \mathcal{KC}$  to the cases $T=I$ and $T=-I$, we obtain
\begin{equation} \label{eq:pos}
C + C^* + C^*C= (I + C)^*(I + C) - I  \in \mathcal{K}(E).
\end{equation}

\begin{equation} \label{eq:neg}
-(C + C^*) + C^*C=(-I + C)^*(-I + C) - I  \in \mathcal{K}(E).
\end{equation}
Multiplying by $\frac{1}{2}$ and adding equations \eqref{eq:pos} and \eqref{eq:neg} yields
$$ C^*C= \frac{1}{2}((C + C^* + C^*C) + (-(C + C^*) + C^*C)) \in \mathcal{K}(E).$$
Now, the canonical image of $C^*C \in \mathcal{K}(E)$ vanishes in $\mathcal{Q}(E)$, and hence
$$ \|\pi(C)\|^2 = \|\pi(C)^*\pi(C)\| = \| \pi(C^*C) \| = \|0\| = 0.$$
This implies $\pi(C) = 0 \in \mathcal{L}(E)/\mathcal{K}(E)$, and therefore $C \in \mathcal{K}(E)$.
\end{proof}

\begin{remark}\label{4uni}
Let $T$ be an arbitrary element in the unit ball of a unital C*-algebra. Then $T=S_1 + i S_2$, where $S_1=\frac{T+T^*}{2}$
and $S_2=\frac{T-T^*}{2i}$ are selfadjoint elements and $T_1=S_1+ i \sqrt{1-S_1^{2}}$, $T_2=S_1- i \sqrt{1-S_1^{2}}$,
$T_3=S_2+ i \sqrt{1-S_2^{2}}$ and $T_4=S_2- i \sqrt{1-S_2^{2}}$ are unitary elements. In particular,
the operator $T$ can be written as a finite linear combination $\sum_{i=1}^{4} \frac{1}{2} \, T_i$ of unitary elements.
\end{remark}

\begin{remark}\label{prime}
The algebra $ \mathcal{A}$ is called {\it prime} if for any two ideals $I$ and $J$ of  $ \mathcal{A}$,
$IJ = \{0\}$ implies $I= \{0\}$ or $J = \{0\}$. In the following theorem we assume that
$\mathcal{L}(E)/\mathcal{K}(E)$ is a prime C*-algebra to utilize
Herstein's theorem (every Jordan automorphism on a prime algebra is either an automorphism or
an anti-automorphism \cite{Her}).
For instance, if $E$ is a Hilbert $\mathcal{K}(H)$-module, then
$\mathcal{L}(E)/\mathcal{K}(E)$ is a prime C*-algebra.
If $E$ is a self-dual Hilbert W*-module, then
$\mathcal{L}(E)/\mathcal{K}(E)$ is a prime C*-algebra,
see \cite{AghaHejazi} for
more detailed information.
\end{remark}

Let $\varphi:\mathcal{L}(E)\rightarrow \mathcal{L}(E)$ be a linear map. Then $\varphi$ is said to be
\textit{unital up to compact operators} if $\varphi(I)=I+K$ for some $K\in \mathcal{K}(E)$, and
$\varphi$ is said to be \textit{surjective up to compact operators} if for every $T$ in $\mathcal{L}(E)$
there exists $S$ in $\mathcal{L}(E)$ such that $\varphi(T)-S \in \mathcal{K}(E)$.

\begin{definition}
We say that a map $\varphi:\mathcal{L}(E)\rightarrow \mathcal{L}(E)$ preserves Kasparov cycles in both directions if
$$ \varphi(T) \in \mathcal{KC}(E) \Longleftrightarrow T \in \mathcal{KC}(E).$$
\end{definition}

\begin{theorem}\label{Uth1}
Let $\varphi:\mathcal{L}(E)\rightarrow \mathcal{L}(E)$  be a linear map which is unital and surjective up to compact operators.
If $ \mathcal{L}(E)/\mathcal{K}(E)$ is a prime C*-algebra, then the following statements are equivalent.
\begin{enumerate}
  \item $\varphi$ preserves Kasparov cycles in both directions.
  \item $\varphi(\mathcal{K}(E)) \subseteq \mathcal{K}(E)$ and the induced map
$\psi:\mathcal{L}(E)/\mathcal{K}(E) \rightarrow \mathcal{L}(E)/\mathcal{K}(E)$,  $\psi  (\pi(T))=\pi(\varphi(T))$ for all
$T \in \mathcal{L}(E),$ is either a continuous $*$-automorphism (i.e.
a continuous isomorphism which is $*$-preserving) or a continuous $*$-anti-automorphism.
\end{enumerate}
\end{theorem}

\begin{proof} It is easy to see that (2) implies (1) by the definition of $\psi$ and the fact that $\psi$
is a unital $*$-isomorphism. The proof of (1) implies (2)  is derived from the following facts.

 We first prove that $\varphi(\mathcal{K}(E)) \subseteq \mathcal{K}(E)$. Let
$T \in \mathcal{K}(E)$, since $\varphi$ is unital and surjective up to compact operators there exist
some compact operators $K$ in $\mathcal{K}(E)$ such that
\begin{eqnarray*}
\varphi(I+T) &=& I+K+\varphi(T),\\
\varphi(I-T) &=& I+K-\varphi(T), ~{\rm and}\\
\varphi(I+iT) &=& I+K+i \varphi(T).
\end{eqnarray*}
The operators $I+T$, $I-T$ and $I+iT$ are Kasparov cycles, so
$I+K+\varphi(T)$, $I+K-\varphi(T)$ and $ I+K+i \varphi(T)$ are Kasparov cycles, too.
Consequently, $\varphi(I+T)^* \varphi(I+T)$, $\varphi(I-T)^* \varphi(I-T)$ and
$\varphi(I+iT)^* \varphi(I+iT)$ are in $I + \mathcal{K}(E)$, which imply
\begin{eqnarray*}
\varphi(T)^* \varphi(T)+ \varphi(T)^*+\varphi(T) & \in & \mathcal{K}(E),\\
\varphi(T)^* \varphi(T)- \varphi(T)^*-\varphi(T) & \in & \mathcal{K}(E), ~{\rm and}\\
\varphi(T)^* \varphi(T)- i \varphi(T)^*+i\varphi(T) & \in & \mathcal{K}(E).
\end{eqnarray*}
We therefore have
\begin{eqnarray*}
\varphi(T)^* \varphi(T) & \in & \mathcal{K}(E),\\
\varphi(T)^* + \varphi(T) & \in & \mathcal{K}(E), ~{\rm and}\\
\varphi(T)^* - \varphi(T) & \in & \mathcal{K}(E),
\end{eqnarray*}
which imply $ \varphi(T) \in \mathcal{K}(E)$, that is, $\varphi(\mathcal{K}(E)) \subseteq \mathcal{K}(E)$.

Since $\varphi(\mathcal{K}(E)) \subseteq \mathcal{K}(E)$
the map $\psi:\mathcal{L}(E)/\mathcal{K}(E) \rightarrow \mathcal{L}(E)/\mathcal{K}(E)$,
$\psi  (\pi(T))=\pi(\varphi(T))$ is well defined and preserves unitary elements. To see
this, we may assume that $ \pi(T)^* \pi(T)= 1_{ \mathcal{L}(E)/\mathcal{K}(E)}=\pi(I)$, since
$\varphi$ preserves Kasparov cycles we find
\begin{eqnarray*}
(\psi( \pi(T)))^*  \psi( \pi(T)) =( \pi ( \varphi(T)))^*  \pi ( \varphi(T)) &=& \pi (  \varphi(T)^* \varphi(T))\\
&=& \pi(I).
\end{eqnarray*}
Similarly $\psi( \pi(T)) (\psi( \pi(T)))^* =\pi(I)$, where $\pi(T) \pi(T)^* = \pi(I)$.

Boundedness of the map $\psi$ is deduced from the preceding fact and Remark \ref{4uni}. Indeed, if
$T \in \mathcal{L}(E)/\mathcal{K}(E)$ and $\|T \| \leq 1$, then $T=\sum_{i=1}^{4} \frac{1}{2} \, T_i$, when each
$T_i$ is a unitary element. Since $\psi$ preserves unitary elements, every $\psi(T_i)$, $i=1, \cdots ,4$ is a unitary
in $\mathcal{L}(E)/\mathcal{K}(E)$ with $\| \psi(T_i) \| \leq 1$, consequently
$$\| \psi(T) \| \leq \frac{1}{2} \sum_{i=1}^{4} \| \psi(T_i) \| \leq 2.$$

We claim that $\psi$ is a $*$-homomorphism or $*$-anti-homomorphism. To see this, let $S$ be a selfadjoint element in
$ \mathcal{L}(E)/\mathcal{K}(E)$, then $e^{itS}$ is a unitary element in $ \mathcal{L}(E)/\mathcal{K}(E)$ for every $t \in \mathbb{R}$.
So far, we know that $\psi$ is a unital continuous linear map which preserves unitary elements, consequently
\begin{eqnarray*}
\pi(I) &=& (\psi(e^{itS}))^* \psi(e^{itS})= (\psi(1+itS+ \frac{(it)^2}{2!}S^2+ \cdots ))^*
\psi(1+it S+ \frac{(it)^2}{2!}S^2+ \cdots )\\
 &=& (1-it(\psi(S))^*- \frac{t^2}{2!} (\psi(S^2))^*+ \cdots ) (1+it\psi(S) - \frac{t^2}{2!}\psi(S^2) + \cdots )\\
 &=& 1+ it(\psi(S) - (\psi(S))^*) -t^2 \, \frac{ \psi(S^2) + (\psi(S^2))^*}{2} + t^2 (\psi(S))^*) \psi(S) + \cdots \,.
\end{eqnarray*}
That is, for every selfadjoint element $S$ in $ \mathcal{L}(E)/\mathcal{K}(E)$ we have
\begin{eqnarray}\label{U5}
(\psi(S))^*=\psi(S),
\end{eqnarray}
and
\begin{eqnarray}\label{U6}
- \frac{1}{2}( \psi(S^2) + (\psi(S^2))^*) + (\psi(S))^*) \psi(S)=0.
\end{eqnarray}
Utilizing the equality (\ref{U5}) and the decomposition $T=S_1+iS_2$ for every
$T \in \mathcal{L}(E)/\mathcal{K}(E)$, where $S_1$ and $S_2$ are selfadjoint, we find $(\psi(T))^*=\psi(T^*)$.
According to equality (\ref{U6}) and the fact that $\psi$ is $*$-preserving,
we get $\psi(S)^2=\psi(S^2)$, for every selfadjoint $S \in \mathcal{L}(E)/\mathcal{K}(E)$.
Since every $T \in \mathcal{L}(E)/\mathcal{K}(E)$ can be written as
$T=S_1+iS_2$, where $S_1$, $S_2$ and $S_1+S_2$ are selfadjoint, we infer
$$ \psi((S_1+S_2)^2)=(\psi(S_1+S_2))^2,$$
which implies
\begin{eqnarray}\label{U7}
 \psi(S_1 S_2+S_2 S_1) = \psi(S_1) \psi(S_2) + \psi(S_2) \psi(S_1),
\end{eqnarray}
and so
\begin{equation}\label{U8}\begin{split}
\psi(T^2) &=\psi(S_1^2) - \psi(S_2^2 ) +i \psi(S_1 S_2+S_2 S_1) \\
          &=\psi(S_1)^2 - \psi(S_2)^2 +i \psi(S_1) \psi(S_2) + i \psi(S_2) \psi(S_1) = \psi(T)^2.
\end{split}
\end{equation}
That is to say, $\psi$ is a Jordan homomorphism.

We now prove that $\psi$ is injective. We are aware that
\begin{eqnarray*}
Ker(\psi) = \{ \pi(T):~ T \in \mathcal{L}(E), ~ \psi( \pi(T))=0 \} &=& \{ \pi(T): ~ T \in \mathcal{L}(E), ~ \pi( \varphi(T))=0 \} \\
&=& \pi( Ker( \varphi)) .
\end{eqnarray*}
Let $C \in Ker( \varphi)$ and $T \in \mathcal{KC}(E)$, then $ \varphi(T+C)= \varphi(T) \in \mathcal{KC}(E)$.
Since $\varphi$ preserves the set of Kasparov cycles in both directions, $T+C \in \mathcal{KC}(E)$, for all
$T\in \mathcal{KC}(E)$.  In view of Lemma \ref{Ulemma1}, the operator $C$ is compact and so $Ker( \varphi) \subseteq \mathcal{K}(E)$.
The latter inclusion implies that $Ker(\psi) = \pi( Ker( \varphi))= \{ 0 \}$.

The above arguments and Herstein's theorem
(every Jordan automorphism on a prime algebra is either an automorphism or an anti-automorphism \cite{Her})
prove that $\psi$  is either a continuous $*$-automorphism or a continuous $*$-anti-automorphism.
\end{proof}


In the rest of this section we utilize Kadison-Ylinen Theorem (\cite[Theorem 3.1]{ylinen})
to characterize unital surjective linear maps preserving the set $ \sigma ( \pi( \cdot ))$.

\begin{lemma}\label{Ulemma2}
Let $\varphi:\mathcal{L}(E) \rightarrow \mathcal{L}(E)$ be an additive map which is surjective up to compact operators. If
$ \sigma( \pi( \varphi (T)) ) = \sigma ( \pi(T))$ for every $T\in \mathcal{L}(E)$, then
$\varphi(\mathcal{K}(E))=\mathcal{K}(E)$.
\end{lemma}
\begin{proof}
Assume that $ \sigma( \pi( \varphi (T)) ) = \sigma ( \pi(T))$ and
there exists a compact operator $C$ such that $\varphi(C)$ is not compact.
Then at least one of $( \varphi(C))^* + \varphi(C)$ and $(\varphi(C))^* - \varphi(C)$ is not compact.
Without loss of generality, we can suppose that $( \varphi(C))^* + \varphi(C)$ is not compact.
We know that
$$\sigma_{l}( \pi(S)) = \sigma_{r}(\pi(S))=\sigma(\pi(S)),$$ for any selfadjoint operator $S$ in $\mathcal{L}(E)$, moreover,
$\sigma (\pi(S)) = {0}$ if and only if $S$ is compact. We therefore have
$\{0\}\neq \sigma ( \pi((\varphi(C))^* + \varphi(C))).$
In view of the  surjectivity of $\varphi$ up to compact operators, there exist
$S \in\mathcal{L}(E)$ and $K \in\mathcal{K}(E)$ such that $\varphi(S) + K = (\varphi(C))^*$, which implies that
\begin{eqnarray*}
\sigma(\pi(S)) = \sigma(\pi(\varphi(S))) = \sigma(\pi(\varphi(S) + K))=
\sigma(\pi((\varphi(C))^*)) &=& (\sigma(\pi(\varphi(C))))^* \\
&=& (\sigma(\pi(C)))^* = \{0\},
\end{eqnarray*}
while
\begin{eqnarray*}
\{0\} \neq \sigma( \pi( \varphi(C) + (\varphi(C))^*)) =
\sigma( \pi( \varphi(C) + (\varphi(C))^*)) &=& \sigma( \pi( \varphi(C+S))) \\
&=& \sigma(\pi(C+S)) =\sigma( \pi(S)),
\end{eqnarray*}
a contradiction. Thus $\varphi(\mathcal{K}(E))\subseteq \mathcal{K}(E)$.

Conversely, suppose that $C \in \mathcal{K}(E)$ and $\varphi(S)= C$,
we have to show that $S$ is a compact operator on $E$.
If on the contrary, $S$ is not compact, then either $S + S^* $ or $S - S^* $ is not compact, say $S + S^* $.
We set $T = \varphi(S^*) $, then
$$\sigma( \pi(T)) = \sigma( \pi(\varphi(S^*))) = \sigma(\pi(S^*)) =
(\sigma(\pi(S)))^* = (\sigma(\pi(\varphi(S))))^*= (\sigma(\pi(C)))^* = \{0\}.$$
Additivity of $\varphi$ implies that
\begin{eqnarray*}
\{0\} \neq \sigma (\pi(S) + (\pi(S))^*) = \sigma(\pi(S +S^*)) &=& \sigma(\pi(\varphi(S + S^*))) \\
 &=& \sigma(\pi(C) + \pi(T)) = \sigma(\pi(T)) = \{0\},
\end{eqnarray*}
which is a contradiction, that is, $\mathcal{K}(E)\subseteq \varphi(\mathcal{K}(E))$.
\end{proof}
\begin{remark}
Let $\varphi:\mathcal{L}(E) \rightarrow \mathcal{L}(E)$ be an additive map which is
surjective up to compact operators and
let $ \sigma_r( \pi( \varphi (T)) ) = \sigma_r ( \pi(T))$ or $ \sigma_l( \pi( (\varphi (T)) ) = \sigma_l ( \pi(T))$ holds
for every $T\in \mathcal{L}(E)$. Following a similar argument, we can establish that $\varphi(\mathcal{K}(E))=\mathcal{K}(E)$.
\end{remark}

A vector space isomorphism is a bijective linear map between two vector spaces that preserves
vector addition and scalar multiplication. A linear map $ \psi : \mathcal{A} \to \mathcal{B}$
between C*-algebras is called \emph{positive} if $ \psi(T)$ is
positive for each positive $T \in \mathcal{A}$. The following theorem is derived from Ylinen's results.
\begin{theorem}\label{ylinen} (Kadison $\&$ Ylinen \cite[Theorem 3.1]{ylinen})
Let $ \psi : \mathcal{A} \to \mathcal{B}$ be a vector isomorphism between unital C*-algebras and
$\psi(1_{ \mathcal{A}})=1_{ \mathcal{B}}$. Then the following two conditions are equivalent
\begin{enumerate}
  \item $ \psi $ is a C*-isomorphism, i.e., $ \psi(T^*)=( \psi(T))^*$ and $ \psi(S^n)=( \psi(S))^n$, for every
  selfadjoint element $S$ and natural number $n$.
  \item $ \psi $ is bipositive, i.e., the linear maps $ \psi $ and $ \psi^{-1} $ are positive.
\end{enumerate}
\end{theorem}

\begin{theorem}\label{Uth2}
Let $\varphi:\mathcal{L}(E)\rightarrow \mathcal{L}(E)$  be a linear map which is unital and surjective up to compact operators.
If $ \mathcal{L}(E)/\mathcal{K}(E)$ is a prime C*-algebra, then the following statements are equivalent.
\begin{enumerate}
  \item $ \sigma( \pi( \varphi (T)) ) = \sigma ( \pi(T))$ for every $T\in \mathcal{L}(E)$.
  \item $\varphi(\mathcal{K}(E)) \subseteq \mathcal{K}(E)$ and the induced map
$\psi:\mathcal{L}(E)/\mathcal{K}(E) \rightarrow \mathcal{L}(E)/\mathcal{K}(E)$,  $\psi  (\pi(T))=\pi(\varphi(T))$ for all
$T \in \mathcal{L}(E),$ is either a continuous $*$-automorphism or a continuous $*$-anti-automorphism.
\end{enumerate}
\end{theorem}

\begin{proof}
It is easy to see that (2) implies (1) by the definition of $\psi$ and the fact that $\psi$
is a unital $*$-isomorphism.

For the converse, by Lemma \ref{Ulemma2}, $\varphi(\mathcal{K}(E)) = \mathcal{K}(E)$, and so
the map $\psi:\mathcal{L}(E)/\mathcal{K}(E) \rightarrow \mathcal{L}(E)/\mathcal{K}(E)$,
$\psi  (\pi(T))=\pi(\varphi(T))$, is well defined. Moreover
\begin{eqnarray*}
\lambda \not \in \sigma( \psi (\pi(T)) ) & \Longleftrightarrow & \lambda \not \in \sigma( \pi(\varphi(T)) )\\
& \Longleftrightarrow & \lambda \not \in \sigma ( \pi(T)).
\end{eqnarray*}
Consequently, $  \sigma( \psi (\pi(T)) ) = \sigma ( \pi(T)) $ for every $T\in \mathcal{L}(E)$.

Continuity of the map $\psi$ is deduced from Theorem 5.5.2 of \cite{Aupetit1}, surjectivity of $\psi$
together with the following equalities
\begin{eqnarray*}
\rho ( \psi (\pi(T)) ) &=& \textrm{max}\, \{ | \lambda |:~ \lambda \in  \sigma( \psi (\pi(T)) ) \} \\
                       &=& \textrm{max}\, \{ | \lambda |:~ \lambda \in  \sigma( \pi(\varphi(T)) \} \\
                       &=& \textrm{max}\, \{ | \lambda |:~ \lambda \in  \sigma( \pi(T)) \}= \rho (\pi(T)).
\end{eqnarray*}

The map $\psi$ is injective. To see this, let $\pi(S)$ be a selfadjoint element in $Ker( \psi )$, then
$$ \| \pi(S) \|= \rho ( \pi(S) )= \rho ( \psi (\pi(S)) )=0,$$ and so $\pi(S)=0$. According to the
decomposition of an operator into the sum of a selfadjoint with an anti-selfadjoint operator, we infer $Ker( \psi )=\{0 \}$.

In view of the equality $  \sigma( \psi (\pi(T)) ) = \sigma ( \pi(T)) $ and Theorem \ref{ylinen}, the linear
map $\psi$ is a C*-isomorphism on the C*-algebra $\mathcal{L}(E)/\mathcal{K}(E)$. We therefore have,
$$ \psi(T^*)=( \psi(T))^*,~ \textrm{for~every~}T \in \mathcal{L}(E)/\mathcal{K}(E), \textrm{~ and}$$
$$ \psi(S^2)=( \psi(S))^2, ~\textrm{for~ every~ selfadjoint~element~} S \in \mathcal{L}(E)/\mathcal{K}(E).$$
By repeating the same arguments as (\ref{U7}) and (\ref{U8}), $\psi$ is a Jordan homomorphism. Utilizing
Herstein's theorem on the prime algebra $\mathcal{L}(E)/\mathcal{K}(E)$, we find that
$\psi$  is either a continuous $*$-automorphism or a continuous $*$-anti-automorphism.

\end{proof}

In operator algebras, the real rank of a C*-algebra is a noncommutative analogue of
Lebesgue covering dimension which was first introduced by Brown and Pedersen.
By definition, a unital C*-algebra $ \mathcal{A}$ has real rank zero if and only if
the invertible self-adjoint elements of $ \mathcal{A}$ are dense in the self-adjoint
elements of $ \mathcal{A}$, cf. \cite[\S V.7]{Davidson}.
The main result of \cite{AghaHejazi}, Theorem 1.2, establishes that if
$\varphi:\mathcal{L}(E)\rightarrow \mathcal{L}(E)$  is a unital, surjective linear map
preserving semi-$ \mathcal{A}$-Fredholm
operators in both directions, and $\mathcal{L}(E)/\mathcal{K}(E)$ is a prime C*-algebra with real rank zero, then
$\varphi(\mathcal{K}(E)) \subseteq \mathcal{K}(E)$ and the induced map
$\tilde{\varphi} :\mathcal{L}(E)/\mathcal{K}(E) \rightarrow \mathcal{L}(E)/\mathcal{K}(E)$,
$\tilde{\varphi}   (\pi(T))=\pi(\varphi(T))$ for all
$T \in \mathcal{L}(E),$ is either an automorphism or an anti-automorphism.
The authors claim to have proven the theorem for semi-$ \mathcal{A}$-Fredholm operators, however,
the paragraph following Definition 3.6 considers the intersection of $\sigma_l$ and $\sigma_r$.
This implies that the proof of the main result may only hold for $ \mathcal{A}$-Fredholm operators.

It is known that the set of Kasparov cycles is a subset of
$ \mathcal{A}$-Fredholm operators. The following theorem extends the main result of  \cite{AghaHejazi}
to the context of Kasparov cycles, showing that the assumption ``$\mathcal{L}(E)/\mathcal{K}(E)$ should be with real rank zero''
is redundant in our results. For $E=H$ a Hilbert space, the Calkin algebra $\mathcal{B}(H)/\mathcal{K}(H)$ is prime,
so our results are of interest alongside those of Mbekhta \cite{Mbekhta1, Mbekhta2} and Hou-Cui \cite{HouCui1}.

\begin{theorem}\label{Uth3}
Let $\varphi:\mathcal{L}(E)\rightarrow \mathcal{L}(E)$  be a linear map which is unital and surjective up to compact operators.
If $ \mathcal{L}(E)/\mathcal{K}(E)$ is a prime C*-algebra, then the following statements are equivalent.
\begin{enumerate}
\item $ \sigma( \pi( \varphi (T)) ) = \sigma ( \pi(T))$ for every $T\in \mathcal{L}(E)$.
\item $\varphi$ preserves Kasparov cycles in both directions.
\item $\varphi(\mathcal{K}(E)) \subseteq \mathcal{K}(E)$ and the induced map
$\psi:\mathcal{L}(E)/\mathcal{K}(E) \rightarrow \mathcal{L}(E)/\mathcal{K}(E)$,  $\psi  (\pi(T))=\pi(\varphi(T))$ for all
$T \in \mathcal{L}(E),$ is either a continuous $*$-automorphism or a continuous $*$-anti-automorphism.
\end{enumerate}
\end{theorem}

\begin{proof}
The results follow immediately from Theorem \ref{Uth1} and Theorem \ref{Uth2}.
\end{proof}

\vspace{1.5em}

\subsection*{Funding} We did not receive any specific grant.
\subsection*{Data Availability}
No data was used.
\subsection*{Acknowledgment.} The author would like to thank Michael Frank (Leipzig) and Ralf Meyer (G\"ottingen)
for thier helpful comments on a preliminary version of this manuscript.


\begin{thebibliography}{99}
\bibitem {AghaHejazi} T. Aghasizadeh and S. Hejazian, Maps preserving semi-Fredholm operators on Hilbert C*-modules,
\textit{ J. Math. Anal. Appl.} {\bf 354} (2009) 625-629.
\bibitem {Aupetit1} B. Aupetit, \textit{ A primer on spectral theory}, Springer-Verlag, New York, 1991.
\bibitem{Black} B. Blackadar, \textit{K-theory for operator algebras},
Mathematical Sciences Research Institute Publications, Cambridge University Press, Cambridge, 1998.

\bibitem{BM} A. Bourhim and M. Mbekhta, Maps preserving the local reduced minimum modulus,
{\it J. Math. Anal. Appl.} {\b 528} (2023), no. 1, Paper No. 127512, 16 pp.
\bibitem {BDF}  L.G. Brown, R.G. Douglas, P.A. Fillmore,
Unitary equivalence modulo the compact operators and extensions of C*-algebras,
{\it Proc. Conf. Operator Theory (Dalhousie Univ. Halifax, N.S. 1973)}, Lecture Notes in Mathematics 345, Springer (1973) 58-128.
\bibitem {bunkeh} U. Bunke and T. Hirschmann, The index of the scattering operator
on the positive spectral subspace, {\it Commun. Math. Phys.} {\bf 148} (1992) 487-502.

\bibitem{Cui1} J. Cui and J. Hou, Additive maps on standard operator algebras preserving parts of the spectrum,
{\it J. Math. Anal. Appl.} {\bf 282} (2003), no. 1, 266-278.



\bibitem{Davidson} K.R. Davidson, \textit{ C*-Algebras by Example}, Field Inst. Monographs, vol. 6, Amer. Math. Soc., Providence, RI, 1996.

\bibitem {DIEU} J. Dieudonn\'{e}, Sur une g\'{e}n\'{e}ralisation du groupe orthogonal \'{a} quatre variables, \textit{ Arch.
Math.} {\bf 1} (1949) 282-287.

\bibitem {FMP} M. Frank, A.S. Mishchenko and A.A. Pavlov, Orthogonality-preserving,
C*-conformal and conformal module mappings on Hilbert
C*-modules, {\it J. Funct. Anal.} {\bf 260} (2011), no. 2, 327-339.

\bibitem {FRMed} M. Frank, C*-submodule preserving module mappings on Hilbert C*-modules,
{\it Mediterr. J. Math.} {\bf23}  (2026), 82.



\bibitem {Her} I.N. Herstein, Jordan homomorphisms, {\it Transactions of the American Mathematical Society}, {\bf 81} (1956), 331-341.

\bibitem {HouCui1} J. Hou and J. Cui, Linear maps preserving essential spectral functions and closeness of operator ranges,
{\it Bull. Lond. Math. Soc.} {\bf39} (2007), no. 4, 575-582.

\bibitem{Jaf1} A.A. Jafarian and  A.R. Sourour, Spectrum-preserving linear maps, \textit{ J. Funct. Anal.} {\bf 66} (1986), 255-261.
\bibitem {KAS} G. Kasparov. The operator K-functor and extensions of C*-algebras. Izv. Akad. Nauk. SSSR Ser. Mat. 44 (1980), 571-636.


\bibitem{Kaplan} I. Kaplansky, \textit{Algebraic and analytic aspects of operator algebras}, CBMS Regional Conference Series in
Math. American Mathematical Society, Providence, 1970.
\bibitem {LAN} E.C. Lance, {\it Hilbert $C^*$-Modules}, LMS Lecture Note
Series 210, Cambridge Univ. Press, 1995.




\bibitem {Mbekhta1} M. Mbekhta,  Linear maps preserving the set of Fredholm
operators, {\it Proc. Amer. Math. Soc.} {\bf135} (2007), no. 11, 3613-3619.
\bibitem{Mbekhta2} M. Mbekhta, Linear maps preserving the minimum and surjectivity moduli of operators,
{\it Oper. Matrices} {\bf 4} (2010), no. 4, 511-518.
						
\bibitem {SharAtiy} K. Sharifi, Atiyah-J\"{a}nich theorem for $\sigma$-C*-algebras,
{\it Appl. Categ. Structures} {\bf 25} (2017), 893-905.


\bibitem{ylinen} K. Ylinen, Vector space isomorphisms of C*-algebras, {\it Studia Math.}  {\bf 46} (1973), 31-34.



\end{thebibliography}
\end{document}